\documentclass{amsart}

\usepackage[all] {xy}
\usepackage[dvips]{graphicx}
\usepackage{graphics,color}
\usepackage{tikz}
\usepackage{amssymb, amsmath, amsbsy} 

\title[Positively Expansive Measures]{Expansive Measures versus  Lyapunov exponents}

\author{A. Armijo and M. J. Pacifico}

\thanks{Partially supported by CAPES, CNPq, FAPERJ}

\newcommand{\ep} {\epsilon}

      \newcommand{\La}{\Lambda}

\newcommand{\N}{\mathbb{N}}

\newcommand{\ov}{\overline}

\newcommand{\D}{{\mathcal Diff}}
\newcommand{\SA}{{\mathcal A}}

\newcommand{\SI}{{\mathcal I}}

\newcommand{\SM}{{\mathcal M}}

\newcommand{\SO}{{\mathcal O}}
\newcommand{\SP}{{\mathcal P}}

\newcommand{\SU}{{\mathcal U}}

\newcommand{\dist}{\operatorname{dist}}

\newcommand{\dime}{\operatorname{dim}}

\newcommand{\Diff}{\operatorname{Diff}}

\newcommand{\Per}{\operatorname{Per}}

\newtheorem{maintheorem}{Theorem}

\newtheorem{mainlemma}{Lemma}

\newtheorem{theorem}{Theorem}
\newtheorem{corollary}[theorem]{Corollary}

\newtheorem{lemma}[theorem]{Lemma}

\newtheorem{Claim}{Claim}
\DeclareMathOperator{\supp}{supp}

\begin{document}

\begin{abstract}
In this paper we investigate the relation between measure expansiveness and hyperbolicity. 
We prove that non atomic invariant ergodic measures with all of its Lyapunov exponents  positive is positively measure-expansive. 
We also prove that local diffeomorphisms robustly positively measure-expansive is expanding.
Finally, we prove that if a $C^1$ volume preserving diffeomorphism that. can not be accumulated by positively measure expansive diffeomorphis have a dominated sppliting.
\end{abstract}

\maketitle

\section{Introduction}

The notion of expansiveness was introduced by Utz in the middle of the twentieth century, see \cite{Ut}.
Roughly speaking, expansiveness means that  orbits through  different points separate when time evolves.
This notion is very important in the context of the theory of dynamical systems and is shared by a large class of dynamical
systems exhibiting chaotic behavior. Nowadays there is an extensive literature about these systems. See, for instance, \cite{Ft,Hi,Le,Vi} and 
references therein.
Examples of expansive systems are hyperbolic diffeomorphisms defined on compact manifolds. 
This includes Anosov systems and the non-wandering set of Axiom A diffeomorphisms. 

Recently it was introduced in \cite{Mo} the notion of 
measure-expansiveness that generalizes the concept of expansiveness.
Roughly speaking, a system is measure is expansive if the set of points
whose orbit is near the orbit of a given point is zero. 
There are already a consistent literature respect measure-expansive systems, relating this property with expansiveness, 
ergodicity and some other properties already established elsewhere.  We refer  to \cite{ArM,PV,BF} and references therein for more on this. 

The purpose of this work is to exploit more this notion and establish some of its relation with expansiveness, existence of some weak form of 
hyperbolicity and existence of positive Lyapunov exponents.

To announce precisely our results, let us introduce some definitions.
To this end, let $(M, d)$ be a compact boundaryless Riemannian manifold and 
$\Diff^1_{loc}(M)$ be the set of $C^1$ local diffeomorphisms $f:M\to M$.
 Let $\mathcal{M}(M)$ be the space of Borel probability measures of $M$.

Recall that $\mu \in \mathcal{M}(M)$   is atomic if there is a point $x \in M $
such that $\mu({x}) > 0$. 
The set of atomic measures $\mathcal{A}(M)$ of $M$ is dense 
 in $\mathcal{M}(M)$. 
A measure $\mu$ is $f$-invariant if $\mu(f^{-1}(B))=\mu(B)$ for
every measurable set $B$ and $\mu$ is ergodic if the measure of any invariant set is zero or one. We denote $\mathcal{M}_f(M)$ for the set of $f$-invariant probability measures on $M$. 

The classical Oseledets's theorem   \cite{Os}, asserts that
for any  
$f$-invariant 
measure $\mu$, 
there exist a 
$Df$-invariant (measurable) splitting $$T_x M=E_1(x) \oplus \cdots \oplus E_{k(x)}(x)$$ and numbers 
${\lambda_1} (x) < {\lambda_2} (x) < \cdots < {\lambda_{k(x)}} (x)$ so that for all $v \in E_i (x) \setminus \{0\}$, it holds
$${\lambda_i} (x) = \displaystyle{\lim_{n\rightarrow \infty} \frac{1}{n} \log \| Df^n (x)v\|}$$
for $\mu$-almost all $x \in M.$ 
The numbers $\lambda_i(x), \, 1\leq i \leq k(x),$ are called the \emph{Lyapunov exponents} of $f$ at $x$.

We say that  $f \in \Diff^1_{loc}(M)$ 
is {\em{positively measure expansive for $\mu$}} (or {\em{positively $\mu$-expansive}} for short) if there is a constant $\delta > 0$ such that $\mu(\Gamma^+_\delta(f, x))=0$ for all $x \in M$, where 

\begin{equation}\label{Gama}
\Gamma^+_\delta(f, x)\equiv\{y\in M\,/\,
d(f^n(x),f^n(y))\leq \delta,\,\, \mbox{for all}\,\,n\in\N\}.
\end{equation}

Note that every positively $\mu$-expansive map is positively
$\mu$-expansive for any $\mu \in \mathcal{M}(M) \setminus \mathcal{A}(M)$.

 The first result 
in this paper establishes that local diffeomorphisms with positive Lyapunov exponents  are   positively $\mu$-expansive. 

\begin{maintheorem}
\label{teoa}
Let $f\in \Diff^1_{loc}(M)$ 
and $\mu\in \mathcal{M}_f\setminus \mathcal{A}(M) $ an 
ergodic probability measure  such that  all of its Lyapunov exponents are 
positive. Then $f$ is positively $\mu$-expansive.
\end{maintheorem}

An open class of dynamical systems such that every non-atomic measure is positively expansive is the class of expanding endomorphisms.
 Recall that a map $f \in \Diff^1_{loc}(M) $
 is {\em{expanding}} if there exists $\beta > 1$ and $K>0$ such that for every $x \in M$ we have 
 $$||Df^n (x)||\geq K\beta^n,\,\ n \in \N.$$
Next we want relate positively $\mu$-expansiveness to expansiveness for maps $f \in \Diff^1_{loc}(M)$.
In this direction we point out that in \cite{Sa} the author proved that the $C^1$-interior of the set of positively measure expansive $C^1$ diffeomorphisms coincides 
with the interior of the set of $C^1$ expansive diffeomorphisms.
In \cite{SSY} it is proved that  the interior of the set of measure-expansive diffeomorphisms coincides with the interior of the expansive diffeomorphisms. 
More recently, in \cite{ LLMS} the authors proved that $C^1$-generically, a differentiable map is positively $\mu$-expansive if and only if it is expanding.
The next result is a version for local diffeomorphisms of \cite[Theorem A]{LLMS} and establishes that local diffeomorphisms $C^1$ robustly positively $\mu$-expansive are expanding.

\begin{maintheorem}\label{teob}
Let $f \in \Diff^1_{loc}(M)$
 and suppose that there is a $C^1$-open neighborhood $\mathcal{U}$ of $f$ such that  
every $g\in \mathcal U$ is  
positively $\mu$-expansive for all $\mu \in \mathcal{M}(M)\setminus \mathcal{A}(M)$.
Then $f$ is expanding.
\end{maintheorem}

To prove this last result, we first prove a perturbing lemma,
Lemma \ref{superlema},  that has the same flavor as \cite[Lemma (1.1)]{F}, establishing that if a map $f$ has a non expanding periodic
point $p$ then there is $\delta> 0$ such that 
$\Gamma_\delta^+(f,p)$ 
contains a manifold $S$ with $\dim(S)\geq 1$.
The proof of the theorem follows by contradiction, applying this lemma to a non expanding periodic point. 
As another consequence of Lemma \ref{superlema} 
we get a similar result to \cite[Theorem 1.1]{ALL}:

\begin{theorem}\label{colo}
There exists an open and dense subset $\mathcal{R}\subset \Diff^1_{loc}(M)$ such that if $f \in \mathcal{R}$ and $f$ is positively $\mu$-expansive, then $f$ is expanding.
\end{theorem}

We can also ask if  measure expansiveness for only the \emph{invariant measures} could be enough to the map have some hyperbolicity. Next, we give a 
partial answer to this question  in the {\em{conservative}} context.
This means that the manifold is endowed with a smooth volume
form $\omega$;
then we can speak of conservative
(i.e.,  volume-preserving) diffeomorphisms. 
In order to announce this last result, 
denote by $\mathcal{D}iff^1_\omega(M)$
the set of $C^1$ conservative diffeomorphisms and to easy notation
set

$$
\SP\SI\SM=\{ f \in \D^1(M); f \, \mbox{is positively $\mu$-expansive for all}\,
\mu \in \SM_f(M)\setminus \SA(M)\}.
$$

\begin{maintheorem}\label{teod}
 Let $ f \in \mathcal{D}iff^1_\omega(M)$ and assume there is an open neighborhood 
 $\SU(f)$ such that $\SU(f) \subset  \SP\SI\SM \cup \{f\}$.
 Then $f$ has a dominated splitting.
\end{maintheorem}

This text is organized as follows.
In Section \ref{s.preliminar} we give the basic definitions and recall previous results proved elsewhere that will be used
to obtain our main theorems. 
In Section \ref{s.teoa} we study positive measure-expansive maps and give a sufficient condition
to a measure-expansive map to be  expanding and prove Theorem~\ref{teoa}.
In Section \ref{s.teob} we give a necessary condition to a local diffeomorphism to be measure-expansive and prove Theorem \ref{teob} and Theorem \ref{colo}.
Finally in Section \ref{s.teod} we prove Theorem \ref{teod}.

{\bf{Acknowledgements.}}\/ The authors thank A. Arbieto for helpful conversations in this subject, and to the anonymous referee by his
comments and suggestions that certainly improved this text.


\section{Preliminaries}\label{s.preliminar}

In this section we set the notation and recall some definitions and results proved else where that we shall use to obtain our results. 
To this  let $M$ be a compact boundaryless $n$-dimensional Riemmanian manifold, $n \geq2$. As above $\Diff^1_{loc}(M)$ 
denotes the set of $C^1$
local diffeomorphisms on $M$ endowed with the $C^1$-topology. 
Denote by $d$ the distance on $M$ induced by the 
Riemannian metric $||\cdot||$ on the tangent bundle $TM$.

Let $f\in \Diff^1_{loc}(M)$
and $p\in M$. Recall that $p$ is a \emph{periodic} point if $f^n(p)=p$ for some $n\geq 1$.  The minimal number $n$ such that $f^n(p)=p$ is the 
\emph{period} of $p$ and is denoted by  $\tau(p)$. 
Given a periodic point $p$ with period $\tau(p)$ we denote by $\SO(p)$ the orbit of $p$, i. e., 
$\SO(p)=\{p,f(p), \ldots, f^{\tau(p)-1}(p)\}.$
 A periodic point $p$ is \emph{hyperbolic} if 
the eigenvalues of $Df^{\tau(p)}(p)$ do not belong to the unit circle $S^1$ and 
is \emph{expanding} if all the eigenvalues of $Df^{\tau(p)}(p)$ have absolute value greater than one.

The stable manifold of $x\in M$, $W^s(x)$,  is defined as 
$$W^s(x)=\{y | d(f^n(y),f^n(x))\rightarrow 0, n \in \N\}.$$

The unstable manifold of $x$, $W^u(x)$, is defined as
 $$W^u(x)=\{y \,|\, d(f^{-n}(y),f^{-n}(x))\rightarrow 0, n \in \N\}.$$

A \emph{saddle point} is a hyperbolic periodic point whose stable and unstable manifolds have a positive dimension.

A compact invariant set $\La$  of a diffeomorphism $f$ is \emph{hyperbolic} if there is a $Df$-
invariant continuous splitting $T_{\Lambda}M=E^s\oplus E^u$ and constants $C > 0$ and $\kappa < 1$ such that for every $x \in \La$ and $n \in \N$
it holds
$$||Df^{-n}(x)_{|E^u_x}||\leq C\kappa^n\mbox{ and }||Df^n(x)_{|E^s_x}||\leq  C\kappa^n.$$ 

A compact $f$-invariant set $\La \subset M$ admits a \emph{dominated splitting}
if the tangent bundle $T_\La M$ has a continuous $Df$-invariant splitting $E_1\oplus \cdots \oplus E_k$ and there exist constants  
$C > 0$,  $0 < \lambda < 1$, such that for all $i<j$,
$\forall x\in\Lambda$ and $n\geq 0$ it holds
$$||Df^n|{E_i(x)}||\cdot ||Df^{-n}|E_j(f^n(x))||\leq C\lambda^n.$$

Let $C^1(M)$ be the set of $C^1$ maps $f: M \to M$, endowed with the $C^1$-topology.
A subset $\mathcal R\subset C^1(M)$ is a \emph{residual subset} if contains a countable intersection of open and 
dense sets.
The countable intersection of residual subsets is also a residual subset.

A property (P) holds \emph{generically} if there exists a residual subset $\mathcal R\subset C^1(M)$ such that any $f\in \mathcal R$
 has the property (P).

We finish this section stating  a lemma, due to V. I. Pliss, whose proof can be found in
\cite[Lemma 11.8, pp 276]{Man}.

\begin{lemma}\label{pliss}
Given $A \geq c_2 > c_1 > 0$, let $\theta_0 = \displaystyle\frac{(c_2 - c_1 )}{(A - c_1 )}$. Then,
given any real numbers $a_1,\ldots , a_N$ such that 
$$\sum^N_{j=1} a_j \geq c_2N \hbox{ and } a_j\leq A \hbox{ for every } 1 \leq j \leq N,$$

there are $l > \theta_0 N$ and $1 < n_1 <\ldots < n_l \leq N$ so that

$$\sum^{n_i}_{j=n+1} a_j \geq c_1(n_i-n) \,\ \hbox{for every} \,\ 1 \leq n \leq n_i\,\ \hbox{and} \,\ i=1,\ldots, l.$$
 
\end{lemma}

\section{Proof of Theorem \ref{teoa}}\label{s.teoa}

In this section we prove Theorem \ref{teoa}. For this, first recall
that a map
$f \in \Diff_{loc}^1(M)$ 
is \emph{asymptotically $c$-expanding} at $x\in M$ if (see \cite[pp 1314]{O})
\begin{equation}\label{condition2}
 \limsup_{n\rightarrow \infty} \frac{1}{n} \sum_{i=0}^{n-1} \log \|Df(f^i(x))^{-1}\|^{-1} > 4c\,.
 \end{equation}
 
 Now, let $\mu$ be an invariant ergodic probability measure such that all of its Lyapunov exponents are positive. 
By  \cite[Lemma 3.5]{O}  there exist $c>0$ and $l\in \N$ such that 
\begin{equation}
\int_{M} \frac{1}{l}\log( \|Df^l(x)^{-1}\|)d\mu <-4c<0.
\label{krerley}
\end{equation}
For every  $l \in \N$ define
the set
$$\mathcal{J}_l:=\{x\in M:\hbox{$f^l$ is assymptotically $c$-expanding at $x$}\}.$$
\begin{Claim}\label{claim1}
The set $\mathcal{J}_l$ has total measure with respect to $\mu$. 
\end{Claim}
In fact,  since $\mathcal J_l$ is $f^l$-invariant and  $\mu$  is ergodic then  the measure of $\mathcal J_l$ is null or total. By contradiction, 
assume that the measure of $\mathcal J_l$ is null.
By Birkhoff's Theorem, the limit in (\ref{condition2})
defines a  measurable and integrable map $\varphi$ that satisfies
 $$\int_M \varphi d\mu= \int_M \log( \|Df^l(x)^{-1}\|^{-1})d \mu.$$
So, $$\int_M \log( \|Df^l(x)^{-1}\|^{-1})d\mu=\int_M \varphi d \mu=\int_{\mathcal J_l} \varphi d\mu + \int_{M \setminus \mathcal J_l} \varphi 
d\mu \, .$$
Since we are assuming $\mu(\mathcal J_l)=0$ we have $\int_{\mathcal J_l} \varphi d\mu=0$
 and thus
$$\int_M \log( \|Df^l(x)^{-1}\|^{-1})d\mu = \int_{M\setminus \mathcal J_l} \varphi d\mu\geq -4c\, .$$
So, by (\ref{krerley}), we get
$$\int_{M} \log( \|Df^l(x)^{-1}\|)d\mu <-4lc \leq -4c\leq\int_{M} \log( \|Df^l(x)^{-1}\|)d\mu\,,$$
a contradiction. Thus the measure of $\mathcal J_l$ is total and finishes the proof of Claim \ref{claim1}.

$\hfill\square$

Now, define $g:=f^l$, since $g$ is a $C^1$-local diffeomorphism, there exists an open cover 
$\{V_i\}_{i \in \Lambda}$ of $M$ such that $g|_{V_i}:V_i \rightarrow g(V_i)$ is a diffeomorphism. 
We can assume that these sets $V_i$ are connected, and by compactness of $M$ there is $\delta'>0$ such that
\begin{equation}\label{numerodelebesgue}
\hbox{$\dist(\xi, \eta)< \delta'$ implies that $\xi, \eta \in V_i$ for some $i \in\Lambda$}.
\end{equation}
Moreover, by uniform continuity there exists $\widehat{\delta}>0$ such that
\begin{equation}
\dist(x,y) < \widehat{\delta} \quad \textrm{ implies that }\quad\frac{\|Dg(x)^{-1}\|}{\|Dg(y)^{-1}\|}>e^{-c/2}.
\label{uniforme}
\end{equation}
Fix $\delta=\min\{\widehat{\delta}, \delta'\}$ and consider 
$\Gamma_\delta^+(f,x)$. 
Then, if $y \in \Gamma_\delta^+(f,x)$ we have that
$$y \in \bigcap_{j \in \N} f^{-j}(B(f^j(x),\delta)) \subset \bigcap_{k \in \N} g^{-k}(B(g^k(x), \delta)).$$
Hence, 
\begin{equation}\label{e.distancia}
\dist(g^k(x),g^k(y))< \delta, \,\ \forall k \in \N.
\end{equation}

Pick a point $x \in M$ satisfying condition (\ref{condition2}) for $g$. Then
$$\limsup_{n\rightarrow \infty} \frac{1}{n} \sum_{i=0}^{n-1}\log \|Dg(g^i(x))^{-1}\|^{-1} > 4c.$$

For $N \in \N$ sufficiently large we have 
\begin{equation}\label{NPliss}
\frac{1}{N} \sum_{i=1}^{N}\log \|Dg(g^i(x))^{-1}\|^{-1} > 4c.
\end{equation}
Applying  Lemma \ref{pliss}  with 
$$A=\displaystyle{\max_{\xi \in M} \log\|Dg(\xi)^{-1}\|^{-1}},\,\ c_2=4c, \,\ c_1=2c,\,\,\, \mbox{and} \,\,\,\theta_0=\frac{2c}{A-2c},$$ 
we obtain that there are $l \in \N$, with $l>\theta_0N,$ and $1<n_1<\ldots < n_l<N$, with $N$  as in (\ref{NPliss}), such that
$$ \sum_{j=n+1}^{n_i}\log \|Dg(g^j(x))^{-1}\|^{-1} > (n_i-n)2c
\,\ \hbox{for every} \,\ 1 \leq n \leq n_i\,\ \hbox{and} \,\ i=1,\ldots, l.$$

Moreover, for all $z$ satisfying $d(z,g^{j}(x))<\delta$, for all $j\in \N$,  (\ref{uniforme}) implies that
$$\frac{\|Dg(g^{j}(x))^{-1}\|}{\|Dg(z)^{-1}\|}>e^{\frac{-c}{2}}.$$
Then 
\begin{equation}
\|Dg(z)^{-1}\|<e^{\frac{c}{2}}\|Dg(g^{j}(x))^{-1}\|.
\label{norma}
\end{equation}
\\
By connectedness, (\ref{numerodelebesgue}) and the Mean Value Theorem applied to a inverse branch $g^{-1}$ 
which sends $g^{n+1}(x)$ to $g^n(x)$, for  $n=0,\ldots,n_i$,we get
\begin{eqnarray*}
\dist(g^{n}(x),g^{n}(y)) & = & \dist(g^{-1}(g^{n+1}(x)),g^{-1}(g^{n+1}(y)))\\
& \leq & \|Dg(z_n)^{-1}\|\dist(g^{n+1}(x),g^{n+1}(y)),
\end{eqnarray*}
with  $\dist(z_n,g^{n+1}(x))<\delta$. Thus, by (\ref{norma}) we get
\begin{eqnarray*}
 \dist(x,y) & \leq & \prod_{n=1}^{n_i}\|Dg(g^n(x))^{-1}\|e^{\frac{c}{2}n_i}\delta \\
 & < &\frac{e^{-2c(n_i-1)}}{\|Dg(g(x))^{-1}\|}e^{\frac{c}{2}n_i}\delta \\
 & < & \frac{e^{\frac{-3}{2}c(n_i-\frac{4}{3})}}{\|Dg(g(x))^{-1}\|}\delta.
\end{eqnarray*}

Since $N\theta_0<l<n_l$ we obtain that
$$d(x,y) <\frac{e^{\frac{-3}{2}c(N\theta_0-\frac{4}{3})}}{\|Dg(g(x))^{-1}\|}\delta.$$

As $N$ can be chosen arbitrarily large, we get $\dist(x,y)=0$. Thus $\Gamma_\delta^+(g,x)=\{x\}$ for $\mu$-almost every $x \in M$. Therefore $\mu$ is  positively expansive. This completes the proof of Theorem \ref{teoa}.

$\hfill\square$

\section{Expanding properties and proof of Theorem \ref{teob}}\label{s.teob}

We start observing that a necessary condition to a local diffeomorphism $f$
be positively measure expansive is that for all $\delta > 0$ and all $x \in M$, 
$\Gamma^+_{\delta}(f,x)$ doesn't contain any manifold of dimension greater than or equal to one. 

In fact, assume that $\Gamma^+_{\delta}(f,x)$ contains a manifold $S$, $\dime(S)\geq 1$. 
Let $m$ the Lebesgue measure on $S$ and $\nu$ the  normalized Lebesgue measure on $S$, that is, for any Borel set $A \subset S$, 
$$\nu(A)= \dfrac{m(A)}{m (S)}.$$

Define a measure $\mu$ on $M$ in the following way: for any Borel set $C$ of $M$ we set 
$$\mu(C)=\nu(C \cap S).$$
Clearly $\mu$ is non-atomic and $ \mu(\Gamma^+_{\delta}(f,x)) \geq \mu(S)=1>0$, which implies that $f$ is not positively  measure-expansive.
$\hfill\square$

\subsection{Main Lemma}
Recall that a periodic point $p$ of $f$ with period $\tau(p)$ is \emph{expanding} if all the eigenvalues of $Df^{\tau(p)}(p)$ have absolute value greater than one. The next lemma shows that if there is a non expanding  periodic point $ p $ 
then there is $\delta >0$ such that $\Gamma_\delta^+(f,p)$ contains 
a manifold $S$ with $\dime(S)\geq 1$.

\begin{mainlemma}[Main Lemma]
Let $g \in \Diff^1_{loc}(M)$, $p\in\Per^{\tau(p)}(g)$ such that $Dg^{\tau(p)}(p)$ has at least one eigenvalue $\lambda$ with $|\lambda| \leq 1$, 
and let $\delta >0$. Then there are $h\in \Diff^1_{loc}(M)$ $C^1$-closed to $g$ such that $h =g$ in $\SO(p)$ and a $h^{\tau(p)}$-invariant manifold $\mathcal{I}_p \ni p$ such that
\begin{enumerate}
\item $\mathcal I_p\subset \Gamma_{\delta}^+(h,p)$.
\item If $Dg^{\tau(p)}(p)=id_{T_pM}$, there exists $\mu_h \in \mathcal M_h(M)$ such that $\mu_h(\mathcal I_p)>0$.
\end{enumerate}\label{superlema}
\end{mainlemma}

\begin{proof}
 We endow $M$ with a Riemannian metric. By \cite[Lemma (1.1)]{F}  there exists a local $C^1$-diffeomorphism $h$ $\epsilon$-closed to $g$ 
 in the $C^1$ topology, such that $Dh_x=Dg_x$ for all $x\in \SO(p)$.
 Moreover, 
 there is a $\epsilon$-neighbourhood $\SU$ of the $\SO(p)$ such that
$$h|_{\SO(p) \cup (M \setminus \SU)}=g|_{\SO(p) \cup (M \setminus \SU)}.$$
Furthermore, for all $y \in \SU$ it holds
\begin{equation}\label{ache}
h(y) = \exp_{g^{i+1}(p)}\circ \, Dg(g^i(p))\circ \exp^{-1}_{g^i(p)}(y),
\end{equation}
where $\exp$ is the exponential map.

As consequence of (\ref{ache}) we have 
$$d(h(y),g^{i+1}(p))=||Dg(g^i(p))\exp^{-1}_{g^i(p)}(y)||.$$
 Hence, if $N:=\displaystyle{\max_{y \in M} }||Dg(y)||$ then 
\begin{equation}\label{distancia}
d(h(y),g^{i+1}(p))\leq N||\exp^{-1}_{g^i(p)}(y)||=N d(y,g^i(p)).
\end{equation} 
We define $\ep_i:=\displaystyle{\frac{\ep}{(1+N)^{\tau(p)-i}}}$, $i=0,\dots,\tau(p)-1$. Note that $N \ep_{i-1}< \ep_{i}<\ep_{i+1}.$ 
Let $V:=B_{\ep_0}(p)$ and pick an arbitrary point $x \in V$.

\begin{Claim}
$d(h^i(x),g^i(p))<\ep_i$ for all $i=0,\ldots, \tau(p)-1$.\label{distancia2}
\end{Claim}

The proof goes by induction. 
For $i=0$, by definition of $V$, it holds $d(x,p)<\ep_0$. 
Suppose it is true for  $1< i <\tau(p)-1$. 
Then $$d(h^i(x),g^i(p))<\ep_i,\,\ \text{with}\,\  \ep_i<\ep\quad \mbox{for}\quad 
1< i <\tau(p)-1\,.$$

By (\ref{distancia}) we get
$$d(h^{i+1}(x),g^{i+1}(p))\leq N d(h^i(x),g^i(p))< N\ep_i<\ep_{i+1},$$
completing the induction. This proves the claim.

$\hfill\square$

Since that for any $i=0,\dots, \tau(p)-1$ we have that $d(h^i(x),g^i(p))<\ep$, we can apply (\ref{ache}) and get
$$h(x) = \exp_{g(p)}\circ Dg(p)\circ \exp^{-1}_{p}(x)$$
$$h^2(x) = \exp_{g^{2}(p)}\circ Dg(g(p))\circ \exp^{-1}_{g(p)}(h(x))$$
$$\vdots$$
and 
$$h^{\tau(p)}(x) = \exp_{g^{\tau(p)}(p)}\circ Dg(g^{\tau(p)-1}(p))\circ \exp^{-1}_{g^{\tau(p)-1}(p)}(h^{\tau(p)-1}(x)).$$
Therefore, by the chain rule, 
$$h^{i}(x) = \exp_{g^{i}(p)}\circ Dg^{i}(p)\circ \exp^{-1}_{p}(x)\hbox{, $\forall\, x\in V$, $i\in\{0,\dots,\tau(p)-1\}$}.$$
In particular,
\begin{equation}
h^{\tau(p)}(x) = \exp_{p}\circ Dg^{\tau(p)}(p)\circ \exp^{-1}_{p}(x)\hbox{, $\forall\, x\in V$}.
\label{alfa}
\end{equation}
Now, let $E_p$ be the direct sum of all eigenspaces of the eigenvalues $\lambda$, $|\lambda|\leq 1$ of $Dg^{\tau(p)}(p)$ and  define $\mathcal I_p$ as
$$\mathcal I_p:=\exp_p(E_p)\cap V.$$

\begin{Claim}
$\mathcal I_p$ defined as above is $h^{\tau(p)}$-invariant. 
\end{Claim}
For this, given $x\in \mathcal I_p$ we have $x=\exp_p(v)$ with $v\in E_p$ and $\|v\|<\epsilon_0$, and so
$$h^{\tau(p)}(x)=h^{\tau(p)}(\exp_p(v))=\exp_{p}\circ Dg^{\tau(p)}(p)\circ \exp^{-1}_{p}(\exp_p(v))=
\exp_{p}\circ Dg^{\tau(p)}(p)v.$$
Modifying the Riemannian metric, if necessary, we can suppose that all the eigenspaces of the eigenvalues $\lambda$, with $|\lambda|\leq 1$ are two by two orthogonal. 
In particular, $\|Dg^{\tau(p)}(p)v\|\leq\|v\|$, for all $v\in E_p$. 
Hence, we have that
\[ \displaystyle d(h^{\tau(p)}(x),p)=\|Dg^{\tau(p)}(p)v\|
\leq\|v\|<\epsilon_0,\quad \mbox{and so}\quad h^{\tau(p)}(x)\in V.
\] 
Therefore $h^{\tau(p)}(x)\in\mathcal I_p$ because $E_p$ is invariant by $Dg^{\tau(p)}(p)$. This concludes the proof that  $\mathcal I_p$ is $h^{\tau(p)}$-invariant.

$\hfill\square$

Now we are ready to proof the lemma:
\vspace{0.1cm}

\noindent {\bf{Proof of item (1).}}\/
Given  $\delta>0$ taking $\ep<\delta$ in the definition of $\mathcal I_p$ we get that 

\begin{eqnarray*}
\Gamma_{\delta}^+(h,p) & = & \displaystyle{\bigcap_{i=0}^{\infty} (h^i)^{-1}\left( B_{\delta}\left(h^i(p)\right)\right)}\\
&=&  \bigcap_{k=0}^{\infty}\left(\bigcap_{j=0}^{\tau(p)-1} \left(h^{k\tau(p)+j}\right)^{-1}\left(B_{\delta}\left(h^{k\tau(p)+j}(p)\right)\right)
 \right)\\
 &=& \displaystyle{\bigcap_{k=0}^{\infty}\left(\bigcap_{j=0}^{\tau(p)-1}\left(h^{k\tau(p)}\right)^{-1} \left(\left(h^j\right)^{-1}
\left(B_{\delta}\left(h^{j}(p)\right)\right)\right)\right)}.
 \end{eqnarray*}
By Claim \ref{distancia2}, $V\subset \left(h^j\right)^{-1}\left(B_{\ep_j}\left(h^{j}(p)\right)\right)\subset \left(h^j\right)^{-1}
\left(B_{\delta}\left(h^{j}(p)\right)\right)$ and so
\begin{eqnarray*}
\Gamma^+_{\delta}\left(h,p\right)&=& \displaystyle{\bigcap_{k=0}^{\infty}\left(\bigcap_{j=0}^{\tau(p)-1}\left(h^{k\tau(p)}\right)^{-1}\underbrace
{\left(\left(h^j\right)^{-1}\left(B_{\delta}\left(h^{j}(p)\right)\right)\right)}_{\supset V\supset\mathcal I_p} \right)}\\
&\supset&
\displaystyle{\bigcap_{k=0}^{\infty}\left(\left(h^{k\tau(p)}\right)^{-1}\left(\mathcal I_p\right)\right).}
\end{eqnarray*}
But, since $$h^{\tau(p)}\left(\mathcal I_p\right)\subset\mathcal I_p \hbox{\,\, we get \,\,} \left(h^{\tau(p)}\right)^{-1}\left(\mathcal I_p\right)
\supset\mathcal I_p.$$ 
Thus, $\Gamma^+_\delta\left(h,p\right) \supset\mathcal I_p$, concluding the proof of Item (1).

$\hfill\square$

\noindent {\bf{Proof of item (2).}}\/
Observe that $\mathcal I_p=V$ because $E_p=T_pM$ and
so  (\ref{alfa}) implies that
\begin{equation}
h^{\tau(p)}|_V = id_V 
\label{beta}
\end{equation}
and hence $\left(h^{k\tau(p)}\right)^{-1}|_V = id_V $. Thus, for $\delta'<\ep_0$, $\Gamma^+_{\delta'}(h,p)\subset V$. In fact, we have

\begin{eqnarray*}
\Gamma^+_{\delta'}(h,p)&=&\displaystyle{\bigcap_{k=0}^{\infty}\left(h^{k\tau(p)}\right)^{-1} 
\underbrace{\bigcap_{j=0}^{\tau(p)-1} \left(h^{j}\right)^{-1}\left(B_{\delta'}\left(h^{j}(p)\right)\right)}_{\subset B_{\ep_0}(p)=V}}\\
&=&\displaystyle{\bigcap_{k=0}^{\infty}\left(\bigcap_{j=0}^{\tau\left(p\right)-1} id_V\left( 
\left(h^j\right)^{-1}\left(B_{{\delta'}}\left(h^{j}\left(p\right)\right)\right)\right)\right)}\\
&=&\displaystyle{\bigcap_{k=0}^{\infty}\bigcap_{j=0}^{\tau\left(p\right)-1} \left(h^j\right)^{-1}\left(B_{{\delta'}}\left(h^{j}\left(p\right)
\right)\right)}\\
&=&\displaystyle{\bigcap_{j=0}^{\tau\left(p\right)-1} \left(h^j\right)^{-1}\left(B_{{\delta'}}\left(h^{j}\left(p\right)\right)\right).}\\
\end{eqnarray*}

Then $\Gamma_{{\delta'}}^+\left(h,p\right)=\displaystyle{\bigcap_{i=0}^{\infty} \left(h^i\right)^{-1}\left(B_{{\delta'}}\left(h^i\left(p\right)\right)\right)}
=\displaystyle{\bigcap_{j=0}^{\tau\left(p\right)-1} \left(h^j\right)^{-1}\left(B_{{\delta'}}\left(h^{j}\left(p\right)\right)\right)}$ which is  a finite intersection of open sets and so it is also an 
open set.
Set $V':=\Gamma^+_{{\delta'}}\left(p\right)$.

Now, note that (\ref{beta}) implies that $h\left(A\right)=A$ for all invariant set $A\subset V.$ In particular, $h\left(V'\right)=V' $.

Next we  construct a non-atomic $f-$invariant probability $\mu_h$ on $M$. 
For this, we consider $\mu$ a non-atomic probability measure supported on $\mathcal I_p$. To this end, 
given a Borel set $C\subset M$,  define the measure $\mu_h$ by $$\mu_h\left(C\right)=\frac{1}{\tau\left(p\right)}\sum_{i=0}^{\tau\left(p\right)
-1}\frac{ \mu\left(h^i \left(C \cap V'\right)\right)}{
\mu\left(V'\right)}.$$

Since $h\left(V'\right)=V' $, we obtain
\begin{eqnarray*}
\mu_h\left(M\right)&=& \frac{1}{\tau\left(p\right)}\sum_{i=0}^{\tau\left(p\right)-1} \frac{ \mu\left(h^i \left(M \cap V'\right)\right)}{\mu
\left(V'\right)}
=\frac{1}{\tau\left(p\right)}\sum_{i=0}^{\tau\left(p\right)-1} \frac{ \mu\left(h^i \left( V'\right)\right)}{\mu\left(V'\right)}\\
&=&\frac{1}{\tau\left(p\right)}\sum_{i=0}^{\tau\left(p\right)-1} \frac{ \mu\left( V'\right)}{\mu\left(V'\right)}
= \displaystyle{\dfrac{\tau\left(p\right)}{\tau\left(p\right)}\dfrac{ \mu\left(V'\right)}{\mu\left(V'\right)}=1}.
\end{eqnarray*}

Hence, $\mu_h$ is a probability measure. 
Since 
\begin{eqnarray*}
\mu_h\left(h^{-1}\left(C\right)\right)&=& \frac{1}{\tau\left(p\right)}\sum_{i=0}^{\tau\left(p\right)-1} \frac{\mu\left(h^i \left( h^{-1}
\left(C\right)\cap V'\right)\right)}{\mu\left(V'\right)}\\
&=& \frac{1}{\tau\left(p\right)}\sum_{i=0}^{\tau\left(p\right)-1} \frac{\mu\left( h^{i}\left(C\cap V'\right)\right)}{\mu\left(V'\right)}=\mu_h
(C),
\end{eqnarray*}
we conclude that $\mu_h$ is $h$-invariant.

For $\mathcal I_p$, we obtain
\begin{eqnarray*}
\mu_h\left(\mathcal I_p \right)&=& \frac{1}{\tau\left(p\right)}\sum_{i=0}^{\tau\left(p\right)-1} \frac{\mu\left(h^i \left( \mathcal I_p\cap 
V'\right)\right)} {\mu\left(V'\right)}\\
&=&\frac{1}{\tau\left(p\right)}\sum_{i=0}^{\tau\left(p\right)-1} 
\frac{\mu\left(h^i \left(\mathcal I_p\cap V'\right)\right)}{\mu\left(V'\right)}\geq\frac{1}{\tau\left(p\right)}\frac{\mu\left(\mathcal I_p\cap 
V'\right)}{\mu\left(V'\right)}
\end{eqnarray*}
because $\mathcal I_p\cap V'$ is in this case  an open 
subset of $V' \subset \supp\mu$.
Therefore, we have $\mu_h\left(\mathcal I_p\right)>0$, finishing the proof of item (2).
All together concludes the proof of Lemma \ref{superlema}.
\end{proof}

\begin{corollary}\label{valesemperturbar}
Under the same hypotheses of the lemma, if  $p$ is a periodic hyperbolic saddle then the statement of item 1 can be obtained for $g$ itself without 
perturbations.
\end{corollary}
\begin{proof}
Since $p$ is a periodic hyperbolic saddle, the local stable manifold $W^s_{\varepsilon}\left(p,g\right)$ is a $C^1$-submanifold. We will denote it by 
$I_g$. Hence we can define $\mu=m_{I_g}$ as the induced Lebesgue measure on it, and as before we get that $\mu=m_{I_g}$ is not  an measure expansive, leading to a contradiction.
\end{proof}

\subsection{Proof of Theorem \ref{teob}.}
Let $f$ be a local $C^1$ diffeomorphism robustly positive measure-expansive map.  Then there exists a  neighbourhood $\mathcal U\left(f\right)$ 
such that all of  its elements are positive measure-expansive. Assume, by contradiction, that $f$ is not expanding. 
By \cite[Theorem 1.3]{Ar}, there exists $g \in \mathcal{U}\left(f\right)$, with $p \in \Per\left(g\right)$ of period $\tau\left(p\right)$ such 
that $Dg^{\tau\left(p\right)}\left(p\right)$ has at least one eigenvalue with modulus less  or equal to one.

By  Lemma \ref{superlema}(1), there are $h$  near $g$,  
 $\delta>0$  
and a submanifold $\mathcal I_p \subset \Gamma^+_\delta\left(h,p\right)$,
and a measure $\mu\in \mathcal M\left(M\right) \setminus \mathcal A\left(M\right)$  such that $\mu\left(\mathcal I_p\right)>0$. This 
implies that $h$ is not measure expansive. As $h$ is near $f$, this contradicts that $f$ is robustly positive measure expansive.
This finishes the proof of Theorem \ref{teob}. 

$\hfill\square$

\subsection{Proof of Theorem \ref{colo}}\label{s.propa}

We use the same ideas as in  \cite[Theorem 1.1]{ALL}, but here we consider a dense and open set $\mathcal{R}$ defined by $\mathcal{R}= \mathcal{H} \cup \ov{\mathcal{H}}^c$, where $\mathcal{H}$ is the  set 
 $$\mathcal{H}=\{g \in \Diff_{loc}^1\left(M\right);\,\, g \,\, \mbox{has a saddle hyperbolic periodic point }\,\, q\}.
 $$

First, observe that $\mathcal{H}$ is open:
 given $g \in \mathcal{H}$ there exists a hyperbolic saddle  point $q \in \Per_h\left(g\right)$. 
  Then there is a neighbourhood $\mathcal U$ of $g$ in the $C^1$ topology and a continuous map $p:\mathcal U \to M$ with $p\left(g\right)=q$ 
  such that for $\varphi \in \mathcal U$, $p\left(\varphi\right)$ is a hyperbolic saddle point of $\varphi$. 
 Since the map $\varphi\mapsto D\varphi^{\tau\left(p\left(\varphi\right)\right)}\left( p\left(\varphi\right)\right)$ is  continuous 
 there is a neighbourhood $\mathcal V \subset \mathcal U$ of $g$ such that all eigenvalues of 
 $D\varphi^{\tau(p(\varphi))}\left(p(\varphi)\right)$ do not belong to the unit circle. 
 Hence, for each $\varphi \in  \mathcal V$, $p\left(\varphi\right)$ is a saddle hyperbolic point of $\varphi$. 
 Therefore, $\mathcal V \subset \mathcal H$, concluding that $\mathcal{H}$ is open.
 \vspace{0.1cm}
 
 Since $\mathcal{H}$ is open, we get that  $\mathcal{R}$ defined above is open and dense on $\Diff_{loc}^1\left(M\right)$.
 \vspace{0.1cm}

 Now we claim that if $f \in \mathcal{R}$ and $f$ is $\mu-$expansive for all $\mu \in \mathcal{M}\left(M\right)\setminus \mathcal{A}\left(M\right)$, then  $f$ is 
 expanding. 
 
 The proof  goes by contradiction.
 Assume that $f\in\mathcal R$, $f$ is $\mu$-expansive and $f$ is not expanding. 
 Then \cite[Theorem 1.3]{Ar} implies that there is a sequence $\{f_n\}$ converging to $f$ with periodic points $p_n$ for $f_n$ which have at least 
 one eigenvalue with modulus less or equal to one. Then, by \cite[Lemma (1.1)]{F}, we can find a sequence $\{g_n\}$ converging to $f$ such that 
 $p_n$ is a hyperbolic saddle for every $n$. 
 Hence $f\notin \overline{\mathcal H}^c$.
 Thus $f\in \mathcal H$ and so has a hyperbolic saddle periodic point $p$ and by Corollary \ref{valesemperturbar}, there exists a measure 
 supported in $\mathcal{I}_p$ which is not positively expansive, leading to a contradiction.
 $\hfill\square$

\section{Proof of Theorem \ref{teod} }\label{s.teod}

To prove Theorem \ref{teod}, recall that $\mathcal{PIM}$ is the set  of $C^1$ diffeomorphisms 
$f$ on $M$ that are positively $\mu$-expansive for every $ \mu \in \mathcal{M}_f\left(M\right) \setminus
\mathcal{A}\left(M\right)$,  $\mathcal{M}_f\left(M\right)$ is the set of $f$-invariant measures on $M$ and $\mathcal{A}\left(M\right)$ is the set of atomic measures on $M$.

The proof goes by contradiction. Let $f \in \D^1_\omega(M)$ and $\SU(f)$ be
an open neighborhood of $f$ with
$\mathcal U\left(f\right) \subset \mathcal{PIM} \cup \{f\}$ and
assume that $f$ has no dominated splitting. 
By \cite[Theo. 6]{BDP}, since $f$ has no dominated splitting then there are a
conservative diffeomorphism $g$ in $\mathcal U\left(f\right)$ and a periodic point $p$
of $g$ such that $Dg^{\tau\left(p\right)}\left(p\right)=Id$, where $\tau(p)$ is the period of $p$.

Applying Lemma \ref{superlema} we conclude that there exists a local $C^1$-diffeomorphism $h \in \mathcal U \left(f\right)$ such that it coincides with $g$ in 
the orbit of $p$ and for  every $\delta>0$ there exists a submanifold 
$\mathcal{I}_p\subset \Gamma_{\delta}^+\left(h,p\right)$, $\mathcal I_p\ni p$ and there is a probability $h$-invariant measure $\mu_h$ over $M$ 
such that $\mu_h\left(\mathcal I_p\right)>0$.
Therefore  $\mu_h\left(\Gamma_{\delta}^+\left(h,p\right)\right)>0$ for $\delta$ arbitrary small,
leading to a contradiction (because $h\in\mathcal{PIM}$). This ends the proof of Theorem \ref{teod}.

$\hfill\square$

\vspace{1cm}
\noindent
{\em Alma Armijo}\/
 Departamento de Matem\'atica,
Universidad de las Am\'ericas and
Universidad de Santiago de Chile,
Santiago, Chile. 
\noindent E-mail:  almaarmijo@gmail.com
\vspace{0.2cm}

\noindent{\em M. J. Pacifico }\/
\noindent Instituto de Matem\'atica,
Universidade Federal do Rio de Janeiro,
C. P. 68.530, CEP 21.945-970,
Rio de Janeiro, RJ, Brazil.
\noindent E-mail:   pacifico@im.ufrj.br

\end{document}